\documentclass[11pt]{amsart}

\usepackage[margin=1.15in]{geometry}
\usepackage{amscd,amssymb, mathtools, amsmath, setspace, wasysym}
\usepackage{graphicx}
\usepackage[all]{xy}
\usepackage{url}
\usepackage{hyperref}

\DeclarePairedDelimiter\floor{\lfloor}{\rfloor}

\theoremstyle{plain}
\newtheorem{theorem}{Theorem}[section]
\newtheorem{prop}[theorem]{Proposition}

\newtheorem{cor}[theorem]{Corollary}

\newtheorem{lemma}[theorem]{Lemma}

\theoremstyle{definition}

\newtheorem{defn}[theorem]{Definition}
\newtheorem{rmk}[theorem]{Remark}

\newtheorem*{ex*}{Example}
\newtheorem{problem}[theorem]{Problem}

\newcommand\sO{{\mathcal O}}
\newcommand\sK{{\mathcal K}}
\newcommand\sH{{\mathcal H}}

\newcommand\sP{{\mathcal P}}
\newcommand\sF{{\mathcal F}}
\newcommand\sG{{\mathcal G}}
\newcommand\sE{{\mathcal E}}
\newcommand\sI{{\mathcal I}}

\newcommand\sJ{{\mathcal J}}
\newcommand\sL{{\mathcal L}}

\newcommand{\ddim}{{\rm dim}\,}
\newcommand{\pic}[1]{{\rm Pic}^0(#1)}

\newcommand\rp{{\mathbf{P}}}
\newcommand\rz{{\mathbf{Z}}}
\newcommand\rr{{\mathbf{R}}}

\newcommand\rd{{\mathbf{D}}}

\subjclass[2010]{14K12, 14H45, 14H51}

\title[Theta-regularity of curves and Brill--Noether loci]
{Theta-regularity of curves and Brill--Noether loci}

\author{Luigi Lombardi and Wenbo Niu}

\address{Mathematical Institute\\ University of Bonn, Endenicher Allee 60, 53115, Germany}
 \email{\url{lombardi@math.uni-bonn.de}}

\address{Department of Mathematics\\ University of Notre Dame, 255 Hurley, Notre Dame, IN, 46556}
\email{\url{wniu@nd.edu}}

\begin{document}
\begin{abstract}
We provide a bound on the $\Theta$-regularity of an arbitrary reduced and irreducible 
curve embedded in a polarized abelian variety in terms of its degree and codimension.
This is an ``abelian'' version of Gruson--Lazarsfeld--Peskine's bound on the Castelnuovo--Mumford regularity of a 
non-degenerate curve embedded in a projective 
space. As an application, we provide a Castelnuovo type bound for the genus of a curve in a (non necessarily principally) polarized abelian variety. 
Finally, we bound the $\Theta$-regularity of a class of higher dimensional subvarieties in Jacobian varieties, \emph{i.e.} the 
Brill--Noether loci associated to a Petri general curve, extending earlier work of Pareschi--Popa. 
\end{abstract}

\maketitle

\section{Introduction}
Given a reduced and irreducible non-degenerate curve $C\subset \rp^r$ ($r\geq 3$) of degree $d$ contained in the projective $r$-space, 
natural problems that have been considered by Castelnuovo first (\emph{cf}. \cite[\S3.2 Theorem 1]{Sz}), 
and by Gruson--Lazarsfeld--Peskine later (\cite{GLP}), is 
to see for which integers $n=n(r,d)$ the following properties regarding the geometry of $C$ hold:
\begin{itemize}
\item[$(PA_n)$] The line bundle $\sO_C(n)$ is non-special;\\ 
 \item[$(PB_n)$] The homomorphism $H^0\big(\rp^r,\sO_{\rp^r}(n)\big)\rightarrow H^0\big(C,\sO_C(n)\big)$ is surjective, \emph{i.e.} hypersurfaces of degree $n$ trace 
 out a complete linear system on $C$;\\
 \item[$(PC_n)$] $C$ is cut out in $\rp^r$ by hypersurfaces of degree $n$.
\end{itemize}
In particular, in \cite{Sz} and \cite{GLP} the following effective bounds have been proved:
\begin{flalign*}
 (PA_n) & \quad \mbox{holds \, for \, all }\quad n\geq \floor*{\frac{d-1}{r-1}},\\
 (PB_n) & \quad \mbox{holds \, for \, all }\quad  n\geq d-r+1,\\\
 (PC_n) & \quad \mbox{holds \, for \, all }\quad n\geq d-r+2.
 \end{flalign*}
 
The goal of this paper is to study analogous properties to $(PA)$, $(PB)$, and $(PC)$ for 
integral curves $C$ embedded in a polarized abelian variety $(A,\Theta)$. 
These properties are listed below and take into account the fact that abelian varieties have a ``continuous'' Picard
variety. Hence, as shown for instance
by work of \cite{GL}, \cite{PP1} and \cite{PP3}, it is natural and definitely 
more fruitful to consider the entire family $\{\sO_A(\Theta)\otimes \alpha\}_{\alpha\in {\rm Pic}^0(A)}$ of deformations   
rather than the polarization $\sO_A(\Theta)$ itself.
\begin{itemize}
\item[$(AA_n)$] The cohomology groups $H^1\big(C,\sO_C(n\Theta)\otimes \alpha_{|C}\big)$ vanish for all $\alpha \in \pic{A}$;\\
 \item[$(AA_n^h)$] There exists a closed algebraic subset $V\subset {\rm Pic}^0(A)$ such that ${\rm codim}\,V\geq h$ and
 $H^1\big(C,\sO_C(n\Theta)\otimes \alpha_{|C}\big)=0$ for all $\alpha\notin V$;\\
 \item[$(AB_n)$] The homomorphisms $H^0\big(A,\sO_{A}(n\Theta)\otimes \alpha \big)\rightarrow H^0\big(C,\sO_C(n\Theta)\otimes \alpha_{|C} \big)$ 
 are surjective for all $\alpha\in \pic{A}$;\\
 \item[$(AB_n^h)$] There exists a closed algebraic subset $V\subset {\rm Pic}^0(A)$ such that ${\rm codim}\,V\geq h$ and
 the homomorphisms $H^0\big(A,\sO_{A}(n\Theta)\otimes \alpha\big)\rightarrow H^0\big(C,\sO_C(n\Theta)\otimes \alpha_{|C}\big)$ are 
 surjective for all $\alpha\notin V$;\\
\item[$(AC_n^{{\rm lin}})$] $C$ is cut out by divisors linearly equivalent to $n\Theta$;\\
\item[$(AC_n^{{\rm alg}})$] $C$ is cut out by divisors algebraically equivalent to $n\Theta$.
  \end{itemize}

As in the case of curves embedded in projective spaces where Castelnuovo--Mumford regularity is shown to 
govern properties $(PA)$, $(PB)$, and $(PC)$, so in the case of curves in abelian varieties 
properties $(AA),\ldots, (AC)$ are controlled by a parallel notion of regularity called $\Theta$-\emph{regularity}. 
This notion was introduced by Pareschi--Popa in \cite{PP1} to study generation properties 
of coherent sheaves on abelian varieties and satisfies analogous properties to Castelnuovo--Mumford regular sheaves on projective spaces 
(\emph{cf}. Theorem \ref{riassunto} (iv) and \cite[Theorem 1.8.3]{Laz2}).
More concretely, we say that a closed subvariety $X\subset A$ of a polarized abelian variety $(A,\Theta)$ is $n$-$\Theta$-regular if 
$${\rm codim}_{{\rm Pic}^0(A)} V^i\big(\sI_X((n-1)\Theta)\big) \; > \; i \quad \mbox{ for \, all }\quad i \; > \;0$$
where 
$$V^i\big(\sI_X((n-1)\Theta)\big) \; := \; \big\{\alpha \in {\rm Pic}^0(A) \, | \, 
H^i\big(A,\sI_X((n-1)\Theta)\otimes  \alpha\big) \neq 0\big\}$$ are the non-vanishing loci and $\sI_X$ is the ideal sheaf of $X$.
To have a more global picture regarding computations of $\Theta$-regularity, we 
 mention that this has been only computed for two classes of curves 
 and a few higher-dimensional varieties.
For instance, Abel--Jacobi curves, Abel--Prym curves, 
Brill--Noether loci of type $W_d$ on a smooth curve, and Fano surfaces of lines associated to smooth cubic threefolds 
are $3$-$\Theta$-regular (however there are quite deep differences in these examples as 
both Abel--Jacobi curves and the $W_d$'s satisfy the stronger \emph{strongly} $3$-$\Theta$-\emph{regularity}; see Definition \ref{definition},
\cite[Theorem 4.1]{PP1}, \cite[Theorem A]{CMLV}, \cite[Theorem 1.2]{Ho}).\footnote{A further property that relates the Brill--Noether loci $W_d$ to 
Fano surfaces of lines is that they are both examples of subvarieties representing a minimal cohomology class  
$\frac{[\Theta]^d}{d!}\in H^{2d}(A,\rz)$. Moreover,
it is conjectured that these varieties should be the only ones that represent a minimal class and furthermore that 
strongly $3$-$\Theta$-regularity should provide 
a cohomological characterization for them (\cite[Conjecture 2.2]{PP2}).}
Moreover, a result of Pareschi--Popa, extending to the ``abelian'' case 
a theorem of Bertram--Ein--Lazarsfeld, bounds  the $\Theta$-regularity of a smooth subvariety in terms of the degrees of its defining equations
(\cite[Theorem 6.5]{PP1}).

The main result of this paper 
is a bound on the $\Theta$-regularity of a reduced and irreducible curve on an abelian variety in terms of its codimension and degree. 
\begin{theorem}\label{intr-thm}
If $(A,\Theta)$ is a polarized abelian variety of dimension $r>1$ defined over an algebraically closed field
and $C\subset A$ is a reduced and irreducible curve of degree $d:=\big(\Theta \cdot C\big)={\rm deg}\, \sO_C(\Theta)$,
then $C$ is $(4d+2r+1)$-$\Theta$-regular. 
Moreover, if in addition $\Theta$ is globally generated, then $C$ is 
$(d+r+1)$-$\Theta$-regular.
\end{theorem}
As a consequence, we deduce effective versions of properties $(AA),\ldots, (AC^{{\rm alg}})$.
\begin{cor}\label{intr-cor}
 With notation as in Theorem \ref{intr-thm}, the curve $C$ satisfies properties 
 \begin{gather*}
 (AA_{n}^3), 
 \quad (AA_{n+1}), \quad (AB_{n}^2),  \quad  (AB_{n+1}), \quad (AC_{n+1}^{{\rm lin}}),  \quad  \mbox{and} \quad (AC^{{\rm alg}}_{n}) 
 \end{gather*}
 for all $n\geq 4d+2r$.
 Moreover, if $\Theta$ is globally generated, then we can take $n\geq d+r$.
\end{cor}

Theorem \ref{intr-thm} was inspired 
 by the sharp bound $(d-r+2)$ obtained in \cite[Theorem 1.1]{GLP} for the Castelnuovo--Mumford regularity of a reduced and irreducible 
 non-degenerate curve 
 $C\subset \rp^r$ of degree $d$.
  We do not expect the bounds of Theorem \ref{intr-thm} to be sharp as in the projective case (for instance we do not assume any
 non-degeneracy's type condition), but they do confirm the linearity of the $\Theta$-regularity of curves and, moreover, 
 to the knowledge of the authors, they are the first bounds in this direction.
 
As an application of Corollary \ref{intr-cor}, we deduce a Castelnuovo genus type bound of a reduced and 
irreducible curve in a polarized abelian variety.
Previous results in this direction were given in \cite[Theorem 5.1]{De}, and furthermore in \cite[Theorem B]{PP2} for principally polarized 
abelian varieties (see Remark \ref{bounds}). 
Our bound improves the above ones in a few cases. For instance, 
in the case of abelian surfaces whose N\'{e}ron--Severi group is of rank one and generated by a globally generated line bundle, 
or abelian threefolds of the same type and with $h^0_{\rm min}\geq 6$ (see \eqref{h0min} for the definition of $h_{\rm min}^0$). 
We refer to Theorem \ref{thmbound} for 
the most general version of the following result.
\begin{theorem}\label{intr-cor2}
 Let $C$ be a reduced and irreducible curve of geometric genus $g$ embedded in a polarized abelian variety $(A,\Theta)$ of dimension $r>1$. 
 Denote by $\widetilde{C}$ the normalization of $C$ and by $f:\widetilde{C}\rightarrow A$ the induced morphism. Moreover define the non-negative 
 integer
 \begin{equation}\label{h0min}
 h^0_{\rm min}\; := \; {\rm min}\big\{h^0\big(\widetilde{C},f^*\big(\sO_A(\Theta)\otimes \gamma\big)\big)\, | \, \gamma \in \pic{A}\big\}.
 \end{equation}
 Then the following bound holds 
 $$g \; \leq \; (d-h^0_{{\rm min}}+1) \, (4d+2r+1).$$ Moreover, if in addition $\Theta$ is globally generated, then the bound
  $$g \; \leq\; (d-h^0_{{\rm min}}+1) \, (d+r+1)$$ holds and  
  $h^0_{{\rm min}}  \geq 2$ ($h^0_{{\rm min}}\geq 3$ if $\Theta$ is very ample).
 \end{theorem}

In the second part of the paper (see \S\ref{secwrd}), we provide a first answer to a question raised by Pareschi--Popa \cite[Question 7.21]{PP3}
asking for the least positive integer $n$ such that the Brill--Noether loci
$$W_d^r(C) \; = \; \{L\in {\rm Pic}^d(C)\, | \, h^0(C,L)\geq r+1\}$$ attached to a smooth curve $C$ 
are $n$-$\Theta$-regular (upon the choice of a point on $C$ we can think of these spaces as embedded in the Jacobian $J(C)$).
We bound this integer in the case $C$ is Petri general. As a byproduct we obtain a much better bound for 
the $\Theta$-regularity of Brill--Noether curves which in particular only depends on the codimension.

\begin{theorem}\label{intr-prop}
 Let $C$ be a Petri general curve of genus $g\geq 3$ and let $d,r$ be positive integers such that 
 ${\rm codim}\, W_d^r(C)=(r+1)(g-d+r)\geq 1$.
 Then $W_d^r(C)$ is $$\Big(2+4\floor*{\sqrt{{\rm codim}\, W^r_d(C)}} \Big)\mbox{-}\Theta\mbox{-regular}.$$ 
\end{theorem}

We refer to Theorem \ref{regwrd} and Corollary \ref{corwrd} for more efficient statements. Moreover, we mention that Theorem \ref{intr-prop} 
extends previous work carried out by Pareschi--Popa on 
the $\Theta$-regularity of Brill--Noether loci of type $W_{g-1}^1$ (\cite[Proposition 7.22]{PP3}). 
Finally, in Corollary \ref{bncurves}, we spell out effective bounds to properties $(AA),\ldots ,(AC)$ 
for the case of Brill--Noether curves. 

The strategy of the proof of Theorem \ref{intr-thm} goes as follows.
Similarly to the projective case, the key idea is to construct a resolution (which is exact away an 
algebraic subset of dimension one) 
of the ideal sheaf $\sI_C$ by locally free sheaves 
such that the $\Theta$-regularity of its terms is easy to compute. 
In order to do this, we use regularity theory of sheaves on abelian varieties as developed by Pareschi--Popa in 
\cite{PP1} and \cite{PP3}. More precisely, we notice that a resolution for $\sI_C$ is determined by the study of generation properties of sheaves 
of type 
\begin{equation}\label{intrgamma}
 q_*\big(p^*Q\otimes \sI_{\Gamma}\big)\otimes \sO_A(\Theta)
\end{equation}
where $p$ and $q$ are the projections from $\widetilde{C}\times A$ onto the first and second factor respectively, 
$\widetilde{C}$ is the normalization of $C$, and $\Gamma$ is the graph of the morphism 
$f:\widetilde{C}\rightarrow A$. One of the technical results we prove, which boils down to checking the surjectivity of certain multiplication maps 
on global sections, is that 
the sheaves in \eqref{intrgamma} are 
\emph{continuously globally generated} anytime $Q=B\otimes f^*\sO_A(\Theta)$ 
and $B$ is a general line bundle of degree $g(\widetilde{C})$ on $\widetilde{C}$. In other words, there exist a positive integer $N$ and line bundles 
$\alpha_1,\ldots ,\alpha_N\in {\rm Pic}^0(A)$ such that the sum of twisted evaluation maps
$$\bigoplus_{i=1}^N H^0\big(A,q_*(p^*Q\otimes \sI_{\Gamma})\otimes \sO_A(\Theta)\otimes \alpha_i\big)\otimes \alpha^{-1}_i \, 
\stackrel{\psi}{\longrightarrow}  \, 
q_*(p^*Q\otimes \sI_{\Gamma})\otimes \sO_A(\Theta)$$
is surjective (Proposition \ref{pencil} and Lemma \ref{intersection}).
Therefore, by considering an Eagon-Northcott complex associated to the composition 
of $\psi\otimes \sO_A(-\Theta)$ with the inclusion 
$q_*(p^*Q\otimes \sI_{\Gamma}) \subset H^0(\widetilde{C},Q)\otimes \sO_A$, one 
obtains a resolution of $\sI_C$ which is exact away from $C$ and whose terms are sums of line bundles of type $\sO_A(-i\Theta)\otimes \beta_{i_t}$ 
whit $\beta_{i_t}\in {\rm Pic}^0(A)$ (Proposition \ref{en}).

On the other hand, for the proof of Theorem \ref{intr-prop} we 
exploit the description of Brill--Noether loci as degeneracy loci of a morphism of vector bundles $\varphi:E\rightarrow F$ where 
$E$ is a Picard bundle and $F$ is a sum of topologically trivial line bundles.
It is a general fact that if a degeneracy locus has the expected dimension, then its ideal sheaf admits a resolution
whose terms are tensor products of Schur powers of the two bundles in play (\emph{cf}. for instance \cite{Lay}). 
Therefore, in our situation, the $\Theta$-regularity of $W_d^r(C)$ is computed as soon as one computes 
the $\Theta$-regularity of Picard bundles. This is the goal of Proposition \ref{ppbp} which in particular extends previous calculations of 
Pareschi--Popa on this subject (\cite[Proposition 7.15 and Remark 7.17]{PP3}).

\begin{problem}
 It would be interesting to give a uniform bound for the $\Theta$-regularity of an arbitrary subvariety of an abelian variety in terms 
 of its dimension and degree in order to study 
 higher-dimensional analogue of properties $(AA),\ldots ,(AC)$. In the projective setting the problem of bounding the Castelnuovo--Mumford 
 regularity of subvarieties of projective spaces goes under the name of Eisenbud--Goto's Conjecture (\cite{EG} and 
 \cite[Conjecture 1.8.47]{Laz2}). 
 This has been solved in the case of smooth surfaces and for certain threefolds by means of generic projections 
 (\emph{cf}. \cite{Laz1}, \cite{P}, and \cite{Ran}).
\end{problem}

 \subsection*{Notation} 
 If $\sF$ is a coherent sheaf on a variety $X$ and $D$ a Cartier divisor, we denote by 
 $\sF(D)$ the tensor product $\sF\otimes \sO_X(D)$ of $\sO_X$-modules.

\section{Preliminaries: regularity of sheaves on abelian varieties}\label{secprelim}

Let $(A,\Theta)$ be a polarized abelian variety defined over an algebraically closed field and let $\sF$ be a coherent sheaf on $A$. For 
each integer $i\geq 0$ we define the algebraic closed sets 
$$V^i(\sF) \; := \; \{\alpha\in {\rm Pic}^0(A) \, | \, h^i(A,\sF\otimes \alpha)>0\},$$ usually called \emph{cohomological support loci} or 
\emph{non-vanishing loci}. We recall some terminology from \cite{PP1}.
\begin{defn}\label{definition} 
Let $\sF$ be a coherent sheaf on a polarized abelian variety $(A,\Theta)$.
  \begin{itemize}
  \item[(i).] $\sF$ satisfies $GV$ (\emph{Generic Vanishing}) if ${\rm codim}_{{\rm Pic}^0(A)} \,V^i(\sF)\geq i$ for all $i>0$.\\
  \item[(ii).] $\sF$ is $M$-\emph{regular} (\emph{Mukai regular}) if ${\rm codim}_{{\rm Pic}^0(A)} \,V^i(\sF)>i$ for all $i>0$.\\
  \item[(iii).] $\sF$ satisfies $I.T.0$ (\emph{Index Theorem with index $0$}) if $V^i(\sF)=\emptyset$ for all $i>0$.\\
  \item[(iv).] $\sF$ is $n$-$\Theta$-\emph{regular} if $\sF\big((n-1)\Theta\big)$ is $M$-regular.\\ 
 \item[(v).] $\sF$ is \emph{strongly} $n$-$\Theta$-\emph{regular} if $\sF\big((n-1)\Theta \big)$ satisfies $I.T.0$.\\
 \item[(vi).] $\sF$ is \emph{continuously globally generated} if 
 there exists a positive integer $N$ such that for general line bundles $\alpha_1,\ldots, \alpha_N\in
  {\rm Pic}^0(A)$ the sum of twisted evaluation maps
  $$\bigoplus_{i=1}^N H^0(A,\sF\otimes \alpha_i)\otimes \alpha_i^{\vee} \; \longrightarrow \; \sF$$
 is surjective.\\
 \end{itemize}
\end{defn}
We also recall some properties of $M$-regular and $\Theta$-regular sheaves which we will need later on. 
\begin{theorem}\label{riassunto}
Let $(A,\Theta)$ be a polarized abelian variety and $\sF$ a coherent sheaf on $A$.
 \begin{itemize}
  \item[(i).] If $\sF$ is $M$-regular, then $\sF$ is continuously globally generated.\\
 \item[(ii).] If $\sF$ is continuously globally generated, then $\sF(m\Theta)$ is globally generated for all $m\geq 1$.\\
 \item[(iii).] If $\sF$ is $M$-regular, then $\sF(m\Theta)$ satisfies $I.T.0$ for all $m\geq 1$.\\
 \item[(iv).] If $\sF$ is $0$-$\Theta$-regular, then $\sF$ is globally generated and $m$-$\Theta$-regular for all $m\geq 1$. Moreover, 
 the multiplication maps
 $$H^0\big(A,\sF(\Theta)\big)\otimes H^0\big(A,\sO_A(k\Theta)\big)\; \longrightarrow \; H^0\big(A,\sF((k+1)\Theta)\big)$$ are surjective for all $k\geq 2$.
 \end{itemize}
\end{theorem}
\begin{proof}
 For (i) see \cite[Proposition 2.13]{PP1}; for (ii) see \cite[Proposition 2.12]{PP1};
 for (iii) see \cite[Proposition 2.9]{PP1}; for (iv) see \cite[Theorem 6.3]{PP1}. 
\end{proof}

\section{Regularity of curves in abelian varieties}\label{secproof}

Throughout the whole section we fix the following notation.
Let $A$ be an abelian variety of dimension $r>1$ defined over an algebraically closed field 
and let $\Theta$ be an ample and globally generated divisor on $A$.
Let $\iota:C\hookrightarrow A$ be a reduced and irreducible curve of geometric genus $g$ with
normalization $\nu:\widetilde{C}\rightarrow C$ and set 
$$f \, := \, \iota \circ \nu \, : \, \widetilde{C} \, \longrightarrow \, A \quad \quad  \mbox{ and }\quad \quad d \, := \, \deg \, 
\big( f^*\sO_A(\Theta) \big)\, = \, 
\big(\Theta \cdot C\big).$$
Finally, we denote by $p$ and $q$ the projections from $\widetilde{C}\times A$ onto the first and second factor respectively, and by 
$\Gamma$ the graph of $f$.
 
In order to compute the $\Theta$-regularity of $\sI_C$, we will 
show that it is enough to check the continuous global generation of sheaves of type
$q_* \big(p^*Q\otimes \sI_{\Gamma} \big)\otimes \sO_A(\Theta)$ where $Q$ is a globally generated
line bundle on $\widetilde{C}$ (see Proposition \ref{en}). 
Before doing so, we present a simple lemma along the lines of \cite[Lemma 1.6]{GLP}
computing cohomological support loci of sheaves admitting resolutions which are exact at most away an algebraic closed subset of dimension one.

\begin{lemma}\label{lem}
Let
$$\sE:\quad \cdots \longrightarrow \sE_2\stackrel{d_2}{\longrightarrow} \sE_1\stackrel{d_1}{\longrightarrow} \sE_0
\stackrel{d_0}{\longrightarrow} \sJ\longrightarrow 0$$
be a finite complex of coherent sheaves on a smooth projective irregular variety $Y$ such that $d_0$ is surjective.
If $\sE$ is exact away from an algebraic closed subset of dimension at most one, then there are 
inclusions 
$$V^i(\sJ) \, \subset \, V^i(\sE_0) \, \cup \, V^{i+1}(\sE_1) \, \cup \, \ldots \, \cup \, V^{\ddim Y}(\sE_{\ddim Y-i})\quad \mbox{ for any } 
\quad i \, \geq \, 1.$$
\end{lemma}
\begin{proof}
For any $j\geq 0$ we set
$$\sK_j \, := \, {\rm Ker}\,d_j,\quad \sI_j \, := \, {\rm Im}\, d_j, \quad \mbox{and}\quad 
 \sH_j \, := \, \sK_i/\sI_{j+1}.$$ We note that since $\dim {\rm Supp}\, \sH_j\leq 1$ we must have 
$$V^i(\sH_j) \, = \, \emptyset \quad \mbox{ for all }\quad j \, \geq \, 0\quad \mbox{ and } \quad i \, > \, 1.$$
We prove the lemma by chopping $\sE$ into short exact sequences 
$0\rightarrow \sK_j\rightarrow \sE_j\rightarrow \sI_j\rightarrow 0$ from which we get inclusions 
\begin{eqnarray}\label{inclV1}
V^i(\sI_j) \, \subset \, V^i(\sE_j)\cup V^{i+1}(\sK_j) \mbox{ for all } j\geq 0 \mbox{ and }i\geq 1.
\end{eqnarray}
Moreover, there are exact sequences 
$0\rightarrow \sI_{j+1}\rightarrow \sK_j\rightarrow \sH_j\rightarrow 0$ giving the further inclusions
\begin{eqnarray}\label{inclV2}
V^{i+1}(\sK_j) \, \subset \, V^{i+1}(\sI_{j+1})\cup V^{i+1}(\sH_j) \, = \, V^{i+1}(\sI_{j+1}) \mbox{ for all } j\geq 0 \mbox{ and } i \, \geq \, 1.
\end{eqnarray}
Therefore, by putting together \eqref{inclV1} and \eqref{inclV2}, we obtain inclusions for all $i\geq 1$
$$V^i(\sJ)  \subset  V^i(\sE_0)\cup V^{i+1}(\sE_1)  \cup  \ldots  \cup  V^{\ddim Y}(\sE_{\ddim Y-i})  \cup  
V^{\ddim Y+1}(\sK_{\ddim Y-i}),$$ 
which conclude the proof as $V^{\ddim Y+1}(\sK_{\ddim Y-i})=\emptyset$. 
\end{proof}

The following proposition can be considered as an analogous of \cite[Proposition 1.2]{GLP} in the case of 
abelian varieties.

\begin{prop}\label{en}
Let $Q$ be a globally generated line bundle on $\widetilde{C}$.
If the sheaf $q_*\big(p^*Q\otimes \sI_{\Gamma} \big) \otimes  \sO_A(\Theta)$ is continuously globally generated, then the ideal sheaf
$\sI_C$ is $\big(h^0(\widetilde{C},Q)+r\big)$-$\Theta$-regular.
\end{prop}

\begin{proof}
From the short exact sequence defining $\Gamma$
\begin{equation}\label{gamma}
0\longrightarrow \sI_{\Gamma}\longrightarrow \sO_{\widetilde{C}\times A}\longrightarrow \sO_{\Gamma}=
({\rm id}_{\widetilde{C}},f)_*\sO_{\widetilde{C}}\longrightarrow 0
\end{equation}
we deduce the following sequence 
\begin{equation}\label{ex}
0\longrightarrow q_*(p^*Q\otimes \sI_{\Gamma})\longrightarrow H^0(\widetilde{C},Q)\otimes \sO_A\stackrel{ev}{\longrightarrow} f_*Q
\longrightarrow R^1q_*(p^*Q\otimes \sI_{\Gamma})
\end{equation}
where, since $Q$ is globally generated, the evaluation map $ev$ may fail to be surjective at most at the singular points of $C$.
We denote by $\sG$ the image of $ev$.
Moreover, by hypotheses there exist an integer $N>0$ and line bundles $\alpha_1,\ldots ,\alpha_N\in {\rm Pic}^0(A)$ such that the sum of 
twisted evaluation maps
\begin{equation}\label{cggq}
\bigoplus_{i=1}^N H^0\big(A,q_*(p^*Q\otimes \sI_{\Gamma})\otimes \sO_A(\Theta)\otimes \alpha_i\big)\otimes \alpha_i^{-1}
\longrightarrow q_*(p^*Q\otimes \sI_{\Gamma})\otimes \sO_A(\Theta) 
\end{equation}
is surjective.
For simplicity we define
$$W_i:=H^0\big(A,q_*(p^*Q\otimes \sI_{\Gamma})\otimes \sO_A(\Theta) \otimes \alpha_i\big)\quad \mbox { for } \quad i=1,\ldots, N.$$ 
Therefore, we obtain a presentation of $\sG = {\rm Im}\, ({\rm ev}$):
\begin{equation}\label{succ}
\bigoplus_{i=1}^N W_i\otimes \alpha_i^{-1}
\otimes \sO_A(-\Theta) \stackrel{\varphi}{\longrightarrow} H^0(\widetilde{C},Q)\otimes \sO_A\longrightarrow \sG\longrightarrow 0
\end{equation}
by putting together 
\eqref{ex} and \eqref{cggq}.
Moreover set 
$$E:=\bigoplus_{i=1}^N W_i\otimes \alpha_i^{-1}\otimes \sO_A(-\Theta), \, F:=H^0(\widetilde{C},Q)\otimes \sO_A,\, e:={\rm rank}\, E,\,  
\, a:=\ddim H^0(\widetilde{C},Q)$$
and denote by $\sJ\subset \sO_A$ the $0$-th Fitting ideal of $\sG$. Since $Q$ is globally generated, 
$\sG$ is supported on $C$. Moreover, as $C$ is reduced, there is an inclusion of sheaves 
$\sJ\subset \sI_C$ such that the quotient $\sI_C / \sJ$ may fail to be trivial at most at the (finitely many) singular points of $C$
(see \cite[p. 496]{GLP} for a similar argument). 
As $\sJ$ coincides with the image of $\wedge ^a \varphi$, there is an Eagon--Northcott complex associated to $\varphi:E\rightarrow F$ of 
the form (see \cite[Theorem B.2.2]{Laz2}, or \cite[\S 0]{GLP}): 
\begin{equation*}\label{ENseq}
0\longrightarrow\bigwedge ^{e} E \otimes {\rm Sym}^{e-a}F^{\vee}\longrightarrow \ldots \longrightarrow \bigwedge ^{a+1}E\otimes F^{\vee}
\longrightarrow\bigwedge ^a E\longrightarrow \sJ\longrightarrow 0
\end{equation*}
which is exact away from ${\rm Supp}\,(  {\rm coker}\, \varphi )= C$.
We note that, for all $i\geq 0$, the terms $E_i:=\bigwedge^{a+i} E \otimes {\rm Sym}^{e-a-i}F^{\vee}$ of the resolution can be written as 
$$E_i \, \simeq \, \bigoplus_{t} \sO_A\big((-a-i)\Theta\big)\otimes \beta_{i_t}\quad  \mbox{ for some }\quad \beta_{i_t}\in \pic{A}.$$
Moreover, as $\sI_C/\sJ$ is supported at finitely many points, the sheaf $\sI_C$ is $(a+r)$-$\Theta$-regular as soon as 
$\sJ$ is $(a+r)$-$\Theta$-regular. But an application of
Lemma \ref{lem} to the sequence \eqref{ENseq} provides inclusions of cohomological support loci for all $i>0$
$$V^i \big(\sJ\big((a+r-1)\Theta\big) \big) \, \subset \, \bigcup_{k=i}^{r} V^{k} \big(E_{k-i}\big((a+r-1)\Theta \big)\big)
\, = \, \bigcup_{k,t} V^{k}\big(\sO_A\big((r-k+i-1)\Theta \big) \otimes \beta_{k_t} \big)$$ 
from which we see that  
$$\ddim V^1\big(\sJ\big((a+r-1)\Theta \big) \big) \, = \, 0\quad \mbox{ and }\quad$$
$$V^i\big(\sJ\big((a+r-1)\Theta \big) \big) \, = \, \emptyset\quad \mbox{ for }\quad i \, > \, 1.$$ 
\noindent So in particular
$\sJ$ is $(a+r)$-$\Theta$-regular and the proposition is proved.
\end{proof}

In the next proposition we
will give sufficient conditions for the hypotheses of Proposition \ref{en} to hold (recall that $M$-regular sheaves are continuously 
globally generated by Theorem \ref{riassunto} (i)).

\begin{prop}\label{pencil}
If $B$ is a line bundle on $\widetilde{C}$ such that $f_*B$ is $M$-regular on $A$, then
$q_*\big(p^*(B\otimes f^*\sO_A(\Theta))\otimes \sI_{\Gamma}\big)\otimes \sO_A(\Theta)$ is also $M$-regular.
\end{prop}

\begin{proof}
We continue to denote by $\Gamma$ the graph of the morphism 
$f$, and by $p$ and $q$ the projections from $\widetilde{C}\times A$ onto the first and 
second factor respectively.
We define $$\sH \, := \, q_*\big(\sI_{\Gamma}\otimes p^*(B\otimes f^*\sO_A(\Theta))\big)$$ so that it is enough to check 
${\rm codim}_{\pic{A}}\, V^i\big(\sH\otimes \sO_A(\Theta)\big) >i $ for all $i>0$.

By Theorem \ref{riassunto} the sheaf $f_*B\otimes \sO_A(m\Theta)$ is globally generated and satisfies $I.T.0$
for all $m\geq 1$. Therefore, from the sequence \eqref{gamma} and the isomorphism 
\begin{equation}\label{isopf}
H^0 \big(\widetilde{C}, B\otimes f^*\sO_A(\Theta)\big) \,\simeq \,
H^0 \big(A,f_*B\otimes \sO_A(\Theta) \big)
\end{equation}
obtained with the use of projection formula, we obtain exact sequences 
\begin{gather}\label{exgamma3}
 0\rightarrow \sH \otimes \alpha(\Theta) \longrightarrow
H^0 \big(\widetilde{C},B\otimes f^*\sO_A(\Theta) \big)\otimes \alpha(\Theta) \longrightarrow 
f_*B\otimes \alpha(2\Theta) \longrightarrow 0
\end{gather}
for any $\alpha$ in $\pic{A}$.  
As both $\sO_A(\Theta)$ and $f_*B\otimes \sO_A(2\Theta)$ satisfy $I.T.0$, 
we deduce
$$V^i \big(\sH\otimes \sO_A(\Theta) \big) \, = \, \emptyset\quad \mbox{ for all }\quad i \, \geq \, 2$$ 
so that we are left with the study of the dimension of $V^1\big(\sH\otimes \sO_A(\Theta)\big)$.
To this end we denote by
\begin{equation}\label{mult}
m_{\alpha}: H^0 \big(A,f_*B\otimes \sO_A(\Theta) \big)\otimes H^0 \big(A,\sO_A(\Theta) \otimes \alpha \big)\longrightarrow
H^0 \big(A,f_*B\otimes \sO_A(2\Theta) \otimes \alpha \big)
\end{equation}
the multiplication map induced by \eqref{exgamma3} and \eqref{isopf}, and furthermore we point out the identification
\begin{eqnarray}\label{V1}
V^1\big(\sH\otimes \sO_A(\Theta)\big) \, = \, \{\alpha\in{\rm Pic}^0(A)\,\,\, | \,\,\, m_{\alpha}\mbox{ is not surjective}\}.
\end{eqnarray}
In order to conclude the proof, we claim that it is enough to prove the following inclusion of algebraic sets
\begin{equation}\label{inclusion}
 \big(-1_{\pic{A}}\big)\big(V^1(\sH\otimes \sO_A(\Theta))\big) \, \subset \, V^1(f_*B)
\end{equation}
where $\big(-1_{\pic{A}}\big):{\rm Pic}^0(A)\rightarrow {\rm Pic}^0(A)$ is the isomorphism mapping $\alpha$ to $\alpha^{-1}$.
In fact, this inclusion would immediately imply
$$\ddim V^1\big(\sH\otimes \sO_A(\Theta)\big) \, \leq \, \ddim V^1(f_*B) \, \leq \, r-2$$ as $f_*B$ is $M$-regular by hypotheses. 

Now we show the inclusion \eqref{inclusion}.
Since for any $\alpha\in \pic{A}$ the sheaf $\sO_A(\Theta)\otimes \alpha$ is globally generated, we get exact sequences 
\begin{gather}\label{evkernel}
0\longrightarrow M_{\sO_A(\Theta)\otimes \alpha}\longrightarrow H^0 \big(A,\sO_A(\Theta)\otimes \alpha \big)\otimes \sO_A\longrightarrow
\sO_A(\Theta) \otimes \alpha\longrightarrow 0
\end{gather}
where $M_{\sO_A(\Theta)\otimes \alpha}$ is the kernel bundle of the evaluation map of 
$\sO_A(\Theta)\otimes \alpha$.
Furthermore, by tensorizing \eqref{evkernel} by $\sO_A(\Theta)\otimes \nu_* B$, we get further exact sequences
\begin{gather}\notag 
0\longrightarrow \iota^* \big(M_{\sO_A(\Theta)\otimes \alpha}\otimes \sO_A(\Theta) \big)\otimes \nu_*B
\longrightarrow 
H^0 \big(A,\sO_A(\Theta)\otimes \alpha \big)\otimes \iota^*\sO_A(\Theta) \otimes \nu_*B
\stackrel{{\rm ev}_{\alpha}}{\longrightarrow} \\ \notag 
\iota^*(\sO_A(2\Theta) \otimes \alpha)\otimes \nu_*B \longrightarrow 0.
\end{gather}
We note now that the map induced by $ev_{\alpha}$ on global sections coincides with the multiplication map $m_{\alpha}$ introduced in 
\eqref{mult}. Hence, by using the identification \eqref{V1}, and from the fact that 
\begin{gather}\label{vanishing}
H^1 \big(C,\iota^* \sO_A(\Theta)\otimes \nu_*B \big) \, \simeq \, H^1 \big(A,f_*B\otimes \sO_A(\Theta) \big) \, = \, 0
\end{gather}
(here we used projection formula and the fact that
$f_*B\otimes \sO_A(\Theta)$ satisfies $I.T.0$), we see that 
\begin{equation}\label{incl2}
\alpha\in V^1 \big(\sH \otimes \sO_A(\Theta) \big)\quad \Longrightarrow \quad
H^1\big(C,\iota^* \big(M_{\sO_A(\Theta)\otimes \alpha}\otimes \sO_A(\Theta) \big)\otimes \nu_*B\big) \, \neq \, 0.
\end{equation}

Now pick an arbitrary element $\alpha \in V^1(\sH\otimes \sO_A(\Theta))$ and
set $$W:={\rm Im} \big( H^0 \big(A,\sO_A(\Theta)\otimes \alpha \big) \longrightarrow H^0 \big(C,\iota^*(\sO_A(\Theta) \otimes \alpha) \big)\big).$$ 
We note that $W$ generates the sheaf $\iota^*(\sO_A(\Theta) \otimes \alpha)$. Therefore, by Castelnuovo's pencil trick, 
we can pick two general sections $s_1$ and $s_2$ in $W$ 
such that they generate $\iota^*(\sO_A(\Theta)\otimes \alpha)$.
We have then a commutative diagram:

\centerline{ \xymatrix@=12pt{
 & 0  \ar[r] & \iota^* \big(\sO_A(-\Theta)\otimes  \alpha^{-1} \big)\ar[d]\ar[r] & \sO_{C}\oplus
\sO_{C}\ar[d]\ar[r]^{(s_1,s_2)} & \iota^*\big(\sO_A(\Theta) \otimes \alpha\big) \ar @{=}[d]\ar[r] & 0\\
& 0 \ar[r] & \iota^*M_{\sO_A(\Theta) \otimes \alpha}  \ar[r]  & H^0\big(A,\sO_A(\Theta) \otimes \alpha\big)\otimes \sO_{C} \ar[r] & 
\iota^*\big(\sO_A(\Theta)\otimes
\alpha\big) \ar[r] & 0.\\ }}
\noindent
By defining $$\widetilde{V}:=H^0\big(A,\sO_A(\Theta) \otimes \alpha \big)/ \langle s_1,s_2\rangle,$$ the
Snake Lemma produces an exact sequence
$$0\longrightarrow \iota^*\big(\sO_A(-\Theta) \otimes \alpha^{-1}\big) \longrightarrow \iota^*M_{\sO_A(\Theta) \otimes \alpha}
\longrightarrow \widetilde{V}\otimes \sO_C\longrightarrow 0,$$ and hence the exact sequence
$$0\longrightarrow  \iota^*\alpha^{-1}\otimes \nu_*B
\longrightarrow \iota^*\big(M_{\sO_A(\Theta) \otimes \alpha}\otimes \sO_A(\Theta) \big)\otimes \nu_*B
\longrightarrow \widetilde{V}\otimes \iota^*\sO_A(\Theta) \otimes \nu_*B\longrightarrow 0.$$
Now by \eqref{vanishing} and \eqref{incl2} 
we deduce 
$$H^1\big(C,\iota^*\alpha^{-1}\otimes \nu_*B \big)\; \simeq \; H^1\big(A,f_*B\otimes \alpha^{-1}\big) \; \neq \; 0$$ 
(once again we used projection formula) from which we finally get 
$\alpha^{-1}\in V^1(f_*B)$. Since the line bundle $\alpha\in V^1(\sH\otimes \sO_A(\Theta))$ was arbitrary, 
this concludes the proof of the inclusion \eqref{inclusion}, and hence of the proposition.

\end{proof}

Before proceeding with the proof of Theorem \ref{intr-thm}, we check when the hypotheses of Proposition \ref{pencil} hold. In other words
we establish which line bundles $B$ on the smooth curve $\widetilde{C}$ make the sheaf $f_*B$ $M$-regular on $A$.

\begin{lemma}\label{intersection}
 Let $D$ be a smooth and irreducible curve of genus $g\geq 2$, and let $f:D\rightarrow A$ be a morphism to
an abelian variety $A$ of dimension $r$.
If $B$ is a general line bundle on $D$ of degree $b\geq g-2$, then 
$\ddim V^1(f_*B)\leq g+r-b-2.$ 
In particular, the sheaf $f_*B$ is $M$-regular (\emph{resp.} satisfies $I.T.0$) on $A$ if $B$ is a general line bundle of degree 
$b\geq g$ (\emph{resp.} $b\geq g+r-1$).
\end{lemma}
\begin{proof}
Without loss of generality, we can assume that $f$ in non-constant since in this case 
$V^1(f_*B)=\emptyset$ for any $B\in {\rm Pic}(D)$.
We note that 
the algebraic set $V^1(B)$ is irreducible as Serre duality yields an isomorphism $V^1(B)\simeq W_{2g-2-b}(D)$ 
(here $W_{2g-2-b}(D)$ is the Brill--Noether locus parameterizing line bundles on $D$ of degree $2g-2-b$ with at least one non-zero global section). 
Moreover, the algebraic group $\pic{D}$ acts on itself via translations: for any $\gamma \in \pic{D}$ we write $\gamma V^1(B)$ for 
the image of $V^1(B)$ under the action of $\gamma$. 
Then, by Kleiman's Transversality Theorem \cite[Theorem 2]{Kl}, there exists an open dense subset $V\subset \pic{D}$ such that for all $\gamma \in V$
the fiber product 
$\gamma V^1(B)\times_{\pic{D}}\pic{A}$ is either empty or of dimension $$\ddim V^1(B)+\dim \pic{A}-\dim \pic{D}=\ddim V^1(B)+r-g.$$
As $\gamma V^1(B)\simeq V^1(B\otimes \gamma^{-1})$ for any $\gamma \in V$, we obtain closed 
immersions $V^1(f_*(B\otimes\gamma^{-1}))\hookrightarrow V^1(B\otimes \gamma^{-1})\times_{\pic{D}}\pic{A}$
by the universal property of the fiber product and the projection formula. 
Hence, for any $\gamma\in V$, we have
\begin{eqnarray*}
\dim V^1\big(f_*(B\otimes \gamma^{-1})\big) & \leq & \ddim \big(\gamma V^1(B)\times_{\pic{D}}\pic{A}\big)\\
& = & \ddim V^1(B)+r-g\\
& = & \ddim W_{2g-2-b}(D)+r-g\\
& = & r+g-b-2.
\end{eqnarray*}
\end{proof}

At this point Theorem \ref{intr-thm} follows by putting together the previous lemmas and propositions.

\begin{proof}[Proof of Theorem \ref{intr-thm}.]
We focus first on the case $\Theta$ is globally generated.
By Lemma \ref{intersection} we can pick a line bundle $B$ on $\widetilde C$ of degree $g=g(\widetilde{C})$ such that $f_*B$ is $M$-regular on $A$.
Moreover, we can take $B$ sufficiently general so that $B\otimes f^*\sO_A(\Theta)$ is globally generated.
Hence by Proposition \ref{pencil} 
the sheaf $q_*\big(p^*(B\otimes f^*\sO_A(\Theta))\otimes \sI_{\Gamma}\big)\otimes \sO_A(\Theta)$ is
$M$-regular, and thus continuously globally generated by Theorem \ref{riassunto} (i). 
Therefore, we can apply Proposition \ref{en} with $Q=B\otimes f^*\sO_A(\Theta)$
which tells us that $\sI_C$ is $(h^0(\widetilde{C},Q)+r)$-$\Theta$-regular. But by Riemann--Roch's formula
$$h^0\big(\widetilde{C},B\otimes f^*\sO_A(\Theta)\big)=d+1$$ as 
$h^1(\widetilde{C}, B\otimes f^*\sO_A(\Theta))=0$ by \eqref{vanishing}.

If $\Theta$ is not globally generated, then by Lefschetz's theorem $2\Theta$ is (\cite[p. 57-58]{Mum}). 
Therefore, by the previous argument, $\sI_C\big((r+d')(2\Theta)\big)$ is $M$-regular where $d'={\rm deg}\,\big( f^*\sO_A(2\Theta)\big) 
= 2d$, and hence $\sI_C \big((4d+2r)\Theta \big)$ is $M$-regular. This is equivalent to saying that $\sI_C$ is $(4d+2r+1)$-$\Theta$-regular. 
\end{proof}

\begin{proof}[Proof of Corollary \ref{intr-cor}.]
The statements relative to properties $(AA)$ follow from Theorem \ref{intr-thm} and the fact that, 
for any integer $n\geq 2$, the ideal sheaf of a curve $\sI_C$ is $n$-$\Theta$-regular if and only if 
 properties $(AA_{n-1}^3)$ and $(AB_{n-1}^2)$ hold.
 Moreover, in this case $\sI_C (n\Theta)$ satisfies $I.T.0$ by Theorem \ref{riassunto} (ii),  
 and therefore property $(AA_{n})$ hold.
 
The statements relative to property $(AC^{{\rm lin}})$ follow
again from Theorem \ref{intr-thm} since  
if $\sI_C$ is $n$-$\Theta$-regular, then $\sI_C(n\Theta)$ is $0$-$\Theta$-regular, 
and hence globally generated (Theorem \ref{riassunto} (iv)). 
Finally, the statement relative to property $(AC^{{\rm alg}})$ is an application of the fact that $M$-regular sheaves are continuously globally generated
(Theorem \ref{riassunto} (i)).
 \end{proof}
\section{Genus bounds}\label{secbounds}

The following theorem proves and extends Theorem \ref{intr-cor2}.
 
 \begin{theorem}\label{thmbound}
 Let $(A,\Theta)$ be a polarized abelian variety of dimension $r>1$ defined over an algebraically closed field and  
  let $C\subset A$ be a reduced and irreducible curve of degree $d=\big( \Theta \cdot C\big)$ with normalization $\widetilde{C}$. Denote by
  $f:\widetilde{C}\rightarrow A$ the induced morphism and set
  \begin{gather*}
  h^0_{\rm min}\; := \; {\rm min}\Big\{h^0\Big(\widetilde{C},f^*\big(\sO_A(\Theta)\otimes \gamma\big)\Big)\, | \, \gamma \in \pic{A}\Big\},\\ \notag
 h^0_{\rm max} \; := \; {\rm max}\Big\{h^0\Big(\widetilde{C},f^*\big(\sO_A(\Theta)\otimes \gamma\big)\Big)\, | \, \gamma \in \pic{A}\Big\}.
 \end{gather*}
   Then the following bound holds for the geometric genus $g=g(\widetilde{C})$ of $C$:
  \begin{gather*}
   g \; \leq \; {\rm min}\big\{(d-h^0_{{\rm min}}+1)(4d+2r+1), (d-h^0_{{\rm max}}+1)(4d+2r+2)\big\}.
  \end{gather*}
Moreover, if in addition $\Theta$ is globally generated, then the following bound holds:
\begin{gather*}
 g \; \leq \; {\rm min}\big\{(d-h^0_{{\rm min}}+1)(d+r+1), (d-h^0_{{\rm max}}+1)(d+r+2)\big\},
\end{gather*}
and moreover $h^0_{\rm max}\geq h^0_{{\rm min}}\geq 2$ ($\geq 3$ if $\Theta$ is very ample).
\end{theorem}

\begin{proof}
We prove the theorem in the case $\Theta$ is globally generated by stressing out the parts that should be changed accordingly 
in case $\Theta$ is not globally 
generated.

The fact that $h^0_{{\rm min}}\geq 2$ easily follows since $f$ is a finite morphism and 
$\sO_A(\Theta)\otimes \gamma$ is globally generated and ample for all $\gamma$.
Similarly, the bound $h^0_{{\rm min}}\geq 3$ follows as the restriction of a very ample line bundle is still very ample and there are 
no rational curves in an abelian variety.

 By Corollary \ref{intr-cor} $C$ satisfies $(AA_{d+r+1}^3)$ and therefore 
 there exists a closed subset $V\subset \pic{A}$ such that ${\rm codim}\, V\geq 3$ 
 and the line bundles 
 $\sO_C((d+r+1)\Theta)\otimes \alpha_{|C}$ are non-special for all $\alpha\notin V$. 
 (If $\Theta$ is not globally generated, we would need to work with the bundles $\sO_C \big((4d+2r+1)\Theta \big)\otimes \alpha_{|C}$ 
 instead.) 
 Denoting by $\iota:C\hookrightarrow A$ the closed immersion 
 and by $f:\widetilde{C}\rightarrow A$ the composition of the normalization $\nu:\widetilde{C}\rightarrow C$ with
 $\iota$, we see that the line bundles $f^*\big(\sO_A((d+r+1)\Theta)\otimes \alpha\big)$ are non-special too. This easily follows
 by tensorizing by $\iota^*\big(\sO_{A}((d+r+1)\Theta)\otimes \alpha\big)$ the following short exact sequence 
 $$0\longrightarrow \sO_C \longrightarrow \nu_*\sO_{\widetilde{C}}\longrightarrow Q\longrightarrow 0,$$ 
  where $Q$ is a sheaf supported at the singular points of $C$, and 
 by noting that the induced long exact sequence in cohomology, together with projection formula, yields
 $$H^1\big(\widetilde{C},f^*(\sO_A((d+r+1)\Theta)\otimes \alpha)\big) \simeq H^1\big(C, 
 \iota^* (\sO_A((d+r+1)\Theta)\otimes \alpha)\otimes \nu_*\sO_{\widetilde{C}}\big) = 0.$$
 This immediately gives
 $$\chi\big(f^*\sO_A((d+r+1)\Theta)\otimes \alpha)) \, = \, 
 h^0\big(\widetilde{C},f^*(\sO_A((d+r+1)\Theta)\otimes \alpha)\big)\; \mbox{ for all } \; 
 \alpha \notin V,$$
and moreover, by writing $\alpha \, = \, \gamma^{\otimes d+r+1}$ for some $\gamma \in \pic{A}$, we obtain that for all $\alpha \notin V$ 
\begin{gather*}
h^0\big(\widetilde{C},f^*(\sO_A((d +r +1)\Theta)\otimes \alpha)\big) \; \geq \; \\
(d+r+1) \, h^0(\widetilde{C},f^*(\sO_A(\Theta)\otimes \gamma)) \, - \, (d+r) \, \geq \, (d+r+1)\, h^0_{{\rm min}} \, - \, (d+r) 
\end{gather*}
where the first inequality is \cite[Lemma 5.5]{Ha}.
On the other hand, Riemann--Roch's formula gives 
 $$\chi \big(f^*(\sO_A((d+r+1)\Theta)\otimes \alpha)\big) \, = \, d \, (d+r+1) \, - \, g \, + \, 1,$$ from which we obtain the bound 
 $g\leq (d-h^0_{{\rm min}}+1)(d+r+1)$. (If $\Theta$ is not globally generated one obtains the bound 
 $g\leq (d-h^0_{{\rm min}}+1)(4d+2r+1)$ with the same strategy.)
 
 Now we prove the bound $g\leq (d-h^0_{{\rm max}}+1)(d+r+2)$.
 Exploiting the fact that $C$ satisfies $(AA_{d+r+2})$, this time the line bundles  
 $\sO_C \big((d+r+2)\Theta \big)\otimes \alpha_{|C}$ are non-special for all $\alpha \in \pic{A}$, and therefore so the
 $f^*\big(\sO_A((d+r+2)\Theta)\otimes \alpha\big)$ are. 
 Therefore, by writing $\alpha=\gamma^{\otimes d+r+2}$ for any $\alpha \in \pic{A}$ and by means of \cite[Lemma 5.5]{Ha}, we deduce that for all 
 $\gamma \in \pic{A}$:
  \begin{eqnarray*}
  d(d+r+2) \, - \, g \, +\, 1 & = & \chi\big(f^*\sO_A \big((d+r+2)\Theta \big)\otimes \alpha \big)\\
  & = & h^0 \big(\widetilde{C},f^*(\sO_A \big((d+r+2)\Theta \big)\otimes \alpha)\big)\\
  & \geq & (d+r+2) \, h^0 \big(\widetilde{C},f^*\sO_A(\Theta)\otimes \gamma \big) \, - \, (d+r+1).
  \end{eqnarray*}
  At this point the claimed bound follows at once.
 If $\Theta$ is not globally generated, then $C$ satisfies $(AA_{4d+2r+2})$ and therefore the line bundles 
 $\sO_C \big((4d+2r+2)\Theta \big)\otimes \alpha_{|C}$ are non-special. Therefore one concludes as in the globally generated case.
  \end{proof}

 \begin{rmk}\label{bounds}
 The bounds in Theorem \ref{thmbound} improve, in some particular cases, a previous genus bound 
 $$g(C) \; < \; \frac{(2d-1)^2}{2(r-1)}$$ 
 established by Debarre in \cite[Theorem 5.1]{De}. For instance, this is the case if 
 $A$ is an abelian surface and $\Theta$ is a globally generated line bundle on $A$ such that
 $NS(A) = \rz[\Theta]$. Other examples are provided by 
 polarized abelian threefold $(A,\Theta)$ with $\Theta$ globally generated, $NS(A) = \rz[\Theta]$, and 
 $h^0_{\rm min}\geq 6$.
 Another genus bound which is sharp in the case of principally polarized surfaces can be found in 
 \cite[Theorem B]{PP2}.
\end{rmk}


\section{Regularity of Brill--Noether loci}\label{secwrd}
In this section we work over the field of the complex numbers.
Let $C$ be a smooth curve of genus $g\geq 3$ and denote by 
$$W_d^r(C) \; = \; W_d^r \; := \; \{L\in {\rm Pic}^d(C) \, | \, h^0(C,L)\geq r+1\}$$ the Brill--Noether loci 
endowed with their usual scheme structure given by means 
of Fitting ideals. Up to tensorizing line bundles with a fixed line bundle of degree $-d$, we think of these spaces as subvarieties 
of the Jacobian $\big(J(C),\Theta \big)$ of $C$. 
We aim to bound the $\Theta$-regularity of $W_d^r(C)$ for all $r$ and $d$ in case $C$ is Petri general. 
Before doing so we compute the $\Theta$-regularity of Picard bundles. 

\subsection{Regularity of Picard bundles}
Let $p$ and $q$ be the projections from the product $C\times J(C)$ onto the first and second factor respectively. 
We identify $J(C)$ with its dual abelian variety via the principal polarization $\Theta$
and denote by $\iota:C\hookrightarrow J(C)$ an Abel--Jacobi map. We denote by 
$$\Phi_{\sE}:\rd\big(C\big)\longrightarrow \rd\big(J(C)\big),\quad \sF\mapsto \rr q_*\big(p^*(\sF)\otimes \sE\big)$$ 
the Fourier--Mukai transform defined by the kernel $\sE:=(\iota \times {\rm id}_{J(C)})^*P$ where $P$
is a Poincar\'{e} bundle on $J(C)\times J(C)$. 

An $n$-th \emph{Picard bundle} $E$ ($n\geq 2g-1$) is a locally free sheaf of the form 
$$q_*\big(p^*L\otimes \sE \big)$$ 
where $L$ is a line bundle on $C$ of degree $n$. 
It is straightforward to check that $L$ satisfies $W.I.T.0$ with respect to $\Phi_{\sE}$, 
\emph{i.e.} the cohomology sheaves $H^i(\Phi_{\sE}(L))$ vanish for all $i\neq 0$, and that $E\simeq \widehat{\Phi_{\sE}(L)}:=H^0(\Phi_{\sE}(L))$. 

The regularity of Picard bundles has been studied in \cite[Remark 7.17]{PP3} where the authors give an effective bound on the degree $n$ so that the 
product $E^{\otimes k}\otimes \sO_{J(C)}(\Theta)$ $(1\leq k \leq g-1)$ satisfies $I.T.0$. 
Here we notice that if one is only interested in weaker notions of regularity, such as $M$-regularity or 
$GV$, then it is possible to take smaller values of $n$. 
The proof of the following proposition 
follows the proof of \cite[Proposition 7.15]{PP3}, the hints of \cite[Remark 7.17]{PP3}, and also
part of an argument of \cite[Proposition 4.4]{PP1}.

\begin{prop}\label{ppbp}
   Let $C$ be a smooth curve of genus $g\geq 2$ and let $E$ be an $n$-th Picard bundle.
   If $n\geq 4g-5$ (\emph{resp.} $n\geq 4g-4$), then for any $1\leq k\leq g-1$ the bundles $E^{\otimes k}(\Theta)$ satisfy $GV$ (\emph{resp.} are 
 $M$-regular).
\end{prop}

\begin{proof}
We only prove the $GV$ness of $E^{\otimes k}(\Theta)$ as the other statement is completely analogous.
We can suppose that $E\simeq \widehat{\Phi_{\sE}(L)}$ where $L$ a line bundle of degree $n\geq 4g-5$ on $C$. 
Denote by $\iota : C \hookrightarrow J(C)$
the Abel--Jacobi map $x\mapsto \sO_C(x-x_0)$ where $x_0\in C$ is a point and let $\pi_k:C^k\rightarrow J(C)$ be the morphism $(y_1,\ldots,y_k)\mapsto \iota(y_1)+\ldots +\iota(y_k)$.
First of all we will show that 
\begin{equation}\label{vansim}
R^i\pi_{k*}(L^{\boxtimes k}) \; = \; 0\quad \mbox{ for all }\quad i \; > \; 0.
\end{equation}
Note that $\pi_k =u_k\circ \sigma_k $ where $\sigma_k:C^k\rightarrow C_k$ is the quotient map under the action of the $k$-symmetric group 
on $C^k$, $C_k$ is the $k$-fold symmetric product of $C$, and $u_k:C_k\rightarrow J(C)$ is the Abel--Jacobi map $(x_1,\ldots ,x_k)\mapsto 
\sO_C(x_1+\ldots +x_k-kx_0)$. 
According to \cite[Lemma 4.3.10]{Laz2} in order to show \eqref{vansim} it suffices to show the vanishings of 
$$H^i \big(C^k,L^{\boxtimes k}\otimes \sigma_k^*u_k^*\sO_{J(C)}(l\Theta) \big)\quad \mbox{ for }\quad i>0\quad \mbox{ and }\quad l>>0.$$
But these follow from Kodaira Vanishing Theorem since $L^{\boxtimes k}$ can be written as $\omega_{C^k}\otimes (L')^{\boxtimes k}$ 
where $L'$ is a line bundle on $C$ of positive degree and $\sigma_k^*u_k^*\sO_{J(C)}(l\Theta)$ is a nef line bundle on $C^k$.
Then, as $(L')^{\boxtimes k}$ is ample on $C^k$, the product $(L')^{\boxtimes k}\otimes \sigma_k^*u_k^*\sO_{J(C)}(l\Theta)$ 
is ample for all $l\geq 0$, and hence the desired vanishings follow. 

We now denote by $\Phi_{P}:\rd\big(J(C)\big)\rightarrow   \rd \big(J(C)\big)$ the Fourier--Mukai transform induced by a Poincar\'{e} bundle $P$.
We say that a complex $\sF$ satisfies $W.I.T.i$ (with respect to $\Phi_{P}$) if the cohomology sheaves
$H^j(\Phi_{P}(\sF))$ vanish for all $j\neq i$. 
In this case we denote by $\widehat \sF$ the sheaf such that 
$\Phi_P(\sF)\simeq \widehat{\sF}[-i]$, \emph{i.e.} $\widehat{\sF}=H^i(\Phi_P(\sF))$.
It is easy to check that $\pi_{k*}(L^{\boxtimes k})$ satisfies $I.T.0$, and hence $W.I.T.0$.
In fact, by using the Leray spectral sequence, the vanishings in \eqref{vansim}, and Kodaira vanishing we obtain 
$$H^i\big(J(C),\pi_{k*}(L^{\boxtimes k})\otimes \alpha \big) \, \simeq \, H^i\big(C^k,L^{\boxtimes k}\otimes \pi_{k}^*\alpha \big) \, = \, 0
\,\mbox{ for all }\, \,
\alpha \in {\rm Pic}^0(J(C)), \, \, i>0.$$
Moreover there are isomorphisms
\begin{equation*}\label{Epower}
\widehat {\pi_{k*}(L^{\boxtimes k})} \; \simeq \; \big(m_*((\iota_*L)^{\boxtimes k})\big)^{\widehat{ }} \; \simeq \; 
\widehat{\iota_*L}^{\otimes k} \; = \; E^{\otimes k}
\end{equation*}
where for the second isomorphism we used the exchange formula of Pontrjagin product and tensor product \cite[(3.7)]{Muk} 
(see also \cite[Lemma 7.16]{PP3}).
Therefore by Mukai's inversion theorem \cite[Theorem 2.2]{Muk} the bundle $E^{\otimes k}$ satisfies $W.I.T.g$ and furthermore
$$\widehat{E^{\otimes k}} \; \simeq \; (-1_{J(C)})^*\pi_{k*}(L^{\boxtimes k}).$$ Hence, for any $\alpha \in \pic{J(C)}$, 
we have the following chain of isomorphisms:
\begin{eqnarray*}
 H^i \big(J(C),E^{\otimes k}\otimes \sO_{J(C)}(\Theta)\otimes \alpha \big) & \simeq & {\rm Ext}^i \big(\sO_{J(C)}(-\Theta)\otimes \alpha^{-1},
 E^{\otimes k} \big) \\
 & \simeq & {\rm Ext}^i\big((\sO_{J(C)}(-\Theta)\otimes \alpha^{-1})^{\widehat{ }},\widehat{E^{\otimes k}}\big)\\
 & \simeq & {\rm Ext}^i\big((-1_{J(C)})^*(\sO_{J(C)}(\Theta)\otimes \alpha),(-1_{J(C)})^*\pi_{k*}(L^{\boxtimes k})\big)\\
 & \simeq & H^i\big(J(C),\pi_{k*}(L^{\boxtimes k})\otimes \sO_{J(C)}(-\Theta)\otimes\alpha^{-1}\big)
\end{eqnarray*}
(the third isomorphism is consequence of \cite[Theorem 3.13 (v)]{Muk}).
  
 We conclude that it is enough to prove the inequalities 
\begin{equation}\label{GVcond}
 {\rm codim}_{{\rm Pic}^0(J(C))} \, V^i\big(\pi_{k*}(L^{\boxtimes k})\otimes \sO_{J(C)}(-\Theta)\big) \geq i \quad \mbox{for all}\quad i>0.
\end{equation}
On the other hand, since $R^i\pi_{k*}(L^{\boxtimes k})=0$ for all $i>0$, the vanishing of
$$H^i \big(J(C), \pi_{k*}(L^{\boxtimes k})\otimes \sO_{J(C)}(-\Theta) \otimes \alpha\big)$$ where $\alpha \in {\rm Pic}^0(J(C))$ is equivalent to the 
vanishing of
\begin{equation}\label{van3}
H^i \big(C^k, L^{\boxtimes k}\otimes \pi_k^*\big(\sO_{J(C)}(-\Theta)\otimes \alpha \big)\big).
\end{equation}
Furthermore, as $k\leq g-1$ we recall from \cite[Appendix A.1]{Iz} the isomorphism: 
$$\pi_k^*\sO_{J(C)}(\Theta) \; \simeq \; (\omega_C\otimes A^{-1})^{\boxtimes k}\otimes \sO_{C^k}(-\Delta)$$ where 
$\Delta$ is the union of all diagonal divisors $\Delta_{i,j}:=\{(x_1,\ldots ,x_k) \, | \, x_i=x_j\}\subset  
C^k$ ($1\leq i<j\leq k$) and $A$ is a line bundle of degree $g-k-1$ on $C$.
Then the vanishing of the group in \eqref{van3} is equivalent to the vanishing of
\begin{equation*}\label{SD}
H^{k-i} \big(C^k, (L^{-1}\otimes A^{-1}\otimes \omega_C^{\otimes 2} )^{\boxtimes k}\otimes \sO_{C^k}(-\Delta)\otimes \sigma_k^*u_k^*\alpha^{-1}\big)
\end{equation*}
 where we used Serre duality and wrote $\pi_k = u_k\circ \sigma_k$.

At this point, as the bundle $(L^{-1}\otimes A^{-1}\otimes \omega_C^{\otimes 2})^{\boxtimes k}\otimes \sO_{C^k}(-\Delta)$ is invariant 
under the action of the symmetric group, it can be written as $\sigma_k^* \sL'_{k,E}$ where $\sL'_{k,E}$ is a line bundle on $C_k$ and 
$E$ is a divisor such that $\sO_{C}(E) \simeq L^{-1}\otimes A^{-1}\otimes \omega_C^{\otimes 2}$.
Therefore it is enough to check the vanishings of 
$$H^{k-i} \big(C_k,\sL'_{k,E}\otimes u_k^*\alpha \big)\quad \mbox{ for }\quad i>0$$ which, using again \cite{Iz}, are identified to 
\[
  {\rm Sym}^{k-i} H^1 \big(C,L^{-1}\otimes A^{-1}\otimes \omega_C^{\otimes 2}\otimes \iota^*\alpha \big)
\otimes \wedge ^{i}H^0 \big(C,L^{-1}\otimes A^{-1}\otimes \omega_C^{\otimes 2}\otimes \iota^*\alpha \big) 
\]
for $1\leq i \leq k-1$, and 
\[\wedge ^k H^0 \big(C,L^{-1}\otimes A^{-1}\otimes \omega_C^{\otimes 2}\otimes \iota^*\alpha \big)
  \]
  for $i=k$
(here $\iota^*:{\rm Pic}^0(J(C))\rightarrow {\rm Pic}^0(C)$ is the pull-back isomorphism).

In conclusion, as $\deg (L^{-1}\otimes A^{-1}\otimes \omega_C^{\otimes 2})=3g-3-n+k$, 
we have proved that 
$$V^i\big(\pi_{k*}(L^{\boxtimes k})\otimes \sO_{J(C)}(-\Theta)\big) \; \subset \; W_{3g-3-n+k}(C) \quad \mbox{ for }\quad 
i \, = \, 1,\ldots ,k$$ and 
$$V^i\big(\pi_{k*}(L^{\boxtimes k})\otimes \sO_{J(C)}(-\Theta)\big) \; = \; \emptyset \quad \mbox{ for }\quad i \, > \, k.$$ 
Therefore
$${\rm codim}_{{\rm Pic}^0(J(C))} \, V^i\big(\pi_{k*}(L^{\boxtimes k})\otimes \sO_{J(C)}(-\Theta)\big) \; \geq \; n+3-2g-k 
\quad \mbox{for}\quad i \, = \, 1,\ldots ,k$$
and one can easily check that conditions \eqref{GVcond} are true, for all $1\leq k \leq g-1$, as soon as $n\geq 4g-5$.
\end{proof}

 \begin{rmk}
 We recall that if $n\geq 4g-3$ then $E^{\otimes k}(\Theta)$ satisfies $I.T.0$ for all $1\leq ~ k\leq ~ g-1$ (\cite[Remark 7.17]{PP3}).
 \end{rmk}

\subsection{Regularity of Brill--Noether loci}
We bound the $\Theta$-regularity of the Brill--Noether loci $W_d^r$.
\begin{theorem}\label{regwrd}
Let $C$ be a Petri general curve of genus $g\geq 3$, and let $d>0$ and $r\geq 0$ be integers such that $0\leq \rho:=g-(r+1)(g-d+r)\leq g-1$. 
Write
\begin{equation}\label{divisione}
g \, - \, \rho+\floor*{\sqrt{g-\rho}} \, (4g-r-5)\; = \; k \, (g-1) \, + \, q
\end{equation}
for some integers $k\geq 0$ and $0\leq q < g-1$.
Moreover, set $s=1$ if $q>0$, and $s=0$ else.
Then $\sI_{W_d^r} \big((k+s)\Theta \big)$ is $M$-regular.
\end{theorem}

\begin{proof} 
The idea of the proof is to exploit the fact that ideal sheaves of degeneracy loci of the expected dimension admit an explicit resolution
by locally free sheaves (\emph{cf}. \cite{Lay} and the references therein).
 
 Let $m\geq 2g-d-1$ be an integer. When $C$ is Petri general, the loci $W_d^r$ are 
 $(m+d-g-r)$-th degeneracy loci of a morphism of locally free sheaves
$$\gamma:E\longrightarrow F$$ where $E$ is an $(m+d-g+1)$-th Picard bundle and $F$ is a direct sum of $m$ 
topologically trivial line bundles
on $C$ (\cite[\S VII]{ACGH}). Therefore, from \cite[Theorem on p. $747$ and Remark 1.1]{Lay}, we get a resolution of $\sI_{W_d^r}$
\begin{equation}\label{resolution}
0\longrightarrow K^{g-\rho}\longrightarrow \cdots \longrightarrow K^2\longrightarrow K^1\longrightarrow \sI_{W_d^r}\longrightarrow 0
\end{equation}
where the $K^i$'s are locally free sheaves described as follows.

Denote by $\sP$ the set of non-increasing sequences $I=(i_1,i_2,\ldots)$ such that $i_j\geq i_{j+1}$ for any $j\geq 1$ and $i_j=0$ for $j>>0$.
Denote also by $|I|=\sum_{j} i_j$ the weight of $I$. The Young diagram of a partition $I$ is the set of pairs $(p,q)$ such that 
$p\leq i_q$. We denote by $\bar I$ the adjoint partition to $I$ determined by the involution $(p,q)\mapsto (q,p)$ acting on Young diagrams.
The rank of a partition $I$ is the largest integer $l$ such that the pair $(l,l)$ belongs to the corresponding Young diagram.
If $I$ is a partition of rank $l$, we define the $k$-augmented partition $I(k)$ of $I$ by setting
$$I(k) \, = \, \big(i_1,i_2,\ldots, i_l, \underbrace{l,\ldots ,l}_{k{\rm -times}},i_{l+1}, \ldots\big).$$ 
We denote by $S_I$ the Schur functor associated to a partition $I$.
Finally, by setting $S_0(i):=\{I \in \sP \, | \, |I|=i \}$, the above mentioned sheaves $K^i$ are defined as:
$$K^i \, = \, \bigoplus_{I\in S_0(i)} S_{\bar{I}(m+d-g-r)}E \, \otimes \, S_{I(m+d-g-r)} F^*.$$

We now bound the regularity of the $K^i$'s. First of all, we note that $S_{\bar I(m+d-g-r)} E$ is a direct summand of the bundle 
$E^{\otimes (|\bar I(m+d-g-r)|)}$, and 
similarly so $S_{I(m+d-g-r)} F^*$ is a summand of $({F^*})^{\otimes (|I(m+d-g-r)|)}$. 
Moreover, by the Linearity Formula of \cite{JPW}, we have that $S_{I(m+d-g-r)} F^*$ is still a direct sum of topological trivial line bundles 
as $F^*$ is so.
Choosing the least $m$ so that ${\rm rk}\, E = m+d-g+1\geq 4g-4$, \emph{i.e.} $m=5g-5-d$, by Proposition \ref{ppbp} we have that 
for any $1\leq j \leq g-1$ the sheaves $E^{\otimes j}(\Theta)$ are $M$-regular. Therefore, all we have to do, is to compute the integers 
$|\bar I(m+d-g-r)|=|\bar I(4g-5-r)|$ for any partitions $I$.
But this is an easy calculation as 
$$\max_{ I\in S_0(i)} \big\{|\bar I(4g-5-r)|\big\} \, = \, |I|+ \floor*{\sqrt{|I|}} \, (4g-5-r) 
 \, = \, i + \floor*{ \sqrt{i} } \, (4g-5-r)$$
since $$\max_{I\in S_0(i)} \{{\rm rank}\,(I)\} \, = \, \floor*{ \sqrt{i} }.$$ 

Therefore, by letting $k$ to be the integer as in \eqref{divisione}, we have that 
for all $I\in S_0(i)$ and $i=1,\ldots , g-\rho$
the sheaves $\big(S_{\bar I(4g-5-r)}E\big)\big((k+s)\Theta\big)$ are $M$-regular. 
Moreover, for any index $i$, the tensor products $K^i\big((k+s)\Theta \big)$ are also $M$-regular
as the tensor product of an $M$-regular sheaf with a direct a sum of topologically trivial line bundles is still $M$-regular.
Finally, we conclude that $\sI_{W_d^r}\big((k+s)\Theta\big)$ is $M$-regular
by chopping \eqref{resolution} into short exact sequences.  
\end{proof}

\begin{cor}\label{corwrd}
 Assume the same hypotheses of Theorem \ref{regwrd} and in addition suppose that $r>0$. Then 
 the ideal sheaf $\sI_{W_d^r}$ is $\Big( 2 + 4 \floor*{\sqrt{{\rm codim}\,W_d^r}} \Big)$-$\Theta$-regular.
\end{cor}
\begin{proof}
 We can compute the integers $k$ and $q$ of Theorem \ref{regwrd}:
 $$k \, = \, 4 \floor*{\sqrt{{\rm codim}\,W_d^r}},\quad q \, = \, {\rm codim}\,W_d^r-(r+1)\floor*{\sqrt{{\rm codim}\,W_d^r}} \, \leq \, g-2$$
 (if ${\rm codim}\,W_d^r-(r+1)\floor*{\sqrt{{\rm codim}\,W_d^r}} \, <\, 0$, then we can take an even smaller value for $k$).
The corollary follows from the bound of Theorem \ref{regwrd} and the very definition of $\Theta$-regularity.
 \end{proof}

\begin{rmk}\label{rmkwrd3}
  With the same proof of Theorem \ref{regwrd} it is possible to give similar results about the other notions of regularity, namely $GV$ and $I.T.0$.
 More precisely, if $k$ and $q$ are the quotient and the remainder of the division between 
 $g-\rho+\floor*{\sqrt{g-\rho}} (4g-r-6)$ and $g-1$ respectively, and if $s=s(q)$ is defined as in Theorem \ref{regwrd}, 
 then $\sI_{W_d^r}\big((k+s)\Theta \big)$ satisfies $GV$.
 On the other hand, if $k$ and $q$ are the quotient and remainder of the division between
 $g-\rho+\floor*{\sqrt{g-\rho}} (4g-r-4)$ and $g-1$ respectively, then $\sI_{W_d^r}\big((k+s)\Theta \big)$ satisfies $I.T.0$. 
\end{rmk}

 \subsection{Brill--Noether curves}
 
Brill--Noether curves are Brill--Noether loci of dimension one. 
Brill--Noether curves of the form 
$W_{a+2}^1$ associated to a Petri general curve $C$ of genus $g=2a+1$ are of particular interest. In fact,
Ortega proves that the Jacobian of a general curve of genus $2a+1$ is a Prym--Tyurin variety for the curve $W_{a+2}^1$ (see \cite{Or}).
Here we spell out for which integers $n$ Brill--Noether curves satisfy the properties $(AA_n)$, $(AB_n)$, $(AC_n^{\rm lin})$, 
and $(AC^{\rm alg}_n)$ of the Introduction.
The weaker conditions $(AA_n^h)$ and $(AB_n^h)$ may be worked out as well.
\begin{cor}\label{bncurves}
The sheaf 
 \[
\sI_{W_{a+2}^1}\quad \mbox{ is } \quad 
\begin{dcases}
\Big(2 + 4 \floor*{ \sqrt{2a} } \Big)\mbox{-}\Theta\mbox{-regular} & \mbox{ if } a\geq 3\\  
9\mbox{-}\Theta\mbox{-regular} & \mbox{ if } a=2.
\end{dcases}
\]
Moreover, by setting $E_a:=W_{a+2}^1$ and $e_a:=1+4\floor*{\sqrt{2a}}$, we have:
\begin{enumerate}
 \item The groups $H^1\big(E_a,\sO_{E_a}(n_a\Theta)\otimes \alpha _{|E_a} \big)$ vanish for any $\alpha \in \pic{J(C)}$ 
 and $n_a \geq e_a$.\\
  
 \item  The restriction homomorphisms
 $$H^0\big(J(C),\sO_{J(C)}(n_a\Theta)\otimes \alpha \big) \; \longrightarrow \; H^0\big(E_a, \sO_{E_a}(n_a\Theta)\otimes \alpha_{|E_a} \big)$$ are
 surjective for any $\alpha\in \pic{J(C)}$ and $n_a\geq e_a$.\\

 \item The curve $E_a$ is cut out by divisors linearly equivalent to $n_a\Theta$ for 
$n_2\geq 9$ and $n_a\geq e_a+1$ if $a\geq 3$.\\

\item The curve $E_a$ is cut out by divisors algebraically equivalent to $n_a\Theta$ for  
$n_2\geq 8$ and $n_a\geq e_a$ if $a\geq 3$.

 \end{enumerate}

\end{cor}
\begin{proof}
For the first statement we use Theorem \ref{regwrd} and note that in \eqref{divisione} we can take $k=4\floor*{ \sqrt{2a} }$ and 
$q=2a- \floor*{ \sqrt{2a} }$.
The statements in $(i)$ and $(ii)$ follow as 
$\sI_C(n_a\Theta)$ satisfies $I.T.0$ for $n_a\geq e_a$ by Remark \ref{rmkwrd3}. 
The statements in $(iii)$ and $(iv)$ reduce to check the (continuous) global generation of $\sO_{J(C)}(n_a\Theta)$ for which we apply 
Theorem \ref{riassunto} $(i)$ and $(ii)$. 
\end{proof}

\begin{rmk}\label{regwrd2}
The $\Theta$-regularity bound of the previous corollary should be compared to the weaker bound 
$$\Big(2(2a+1) + 4\frac{(2a+1)!}{a!(a+1)!}+1\Big)\mbox{-}\Theta\mbox{-regular}$$
provided by Theorem \ref{intr-thm}. 
\end{rmk}

 \section*{Acknowledgements}
 We thank Lawrence Ein, Angela Ortega, Giuseppe Pareschi, Mihnea Popa and Matteo Varbaro for useful conversations on the subject. 
We also thank the Mathematics Research Communities (a program of the AMS) for having supported a visit of the second named author by the first, 
and the Mathematics Department of the University of Illinois at Chicago where this work got started.
LL is moreover grateful to Robert Lazarsfeld for the lectures
``The equations defining projective varieties'' he delivered in the program \emph{Felix Klein Lectures} 
 held at the Hausdorff Center for Mathematics (Bonn) in January 2014.
LL was supported by the SFB/TR45 ``Periods, moduli spaces, and arithmetic of algebraic varieties'' of the DFG (German Research Foundation).

\end{document}